\documentclass{amsart}

\usepackage{amssymb}
\usepackage{amsthm}
\usepackage{amsmath}
\usepackage{mathtools}

\usepackage[usenames]{color}

\usepackage{enumitem}

\usepackage{hyperref}
\hypersetup{
    colorlinks=true,
    linkcolor=blue,
    filecolor=magenta,      
    urlcolor=cyan,
    citecolor=red
}

\theoremstyle{plain}
\newtheorem{proposition}{Proposition}[section]
\newtheorem{theorem}[proposition]{Theorem}
\newtheorem{lemma}[proposition]{Lemma}
\newtheorem{corollary}[proposition]{Corollary}
\theoremstyle{definition}

\theoremstyle{remark}
\newtheorem{remark}[proposition]{Remark}

\DeclareMathOperator{\supp}{supp}

\DeclareMathOperator{\SL}{\mathsf{SL}}

\DeclareMathOperator{\SO}{\mathsf{SO}}

\DeclareMathOperator{\End}{End}

\DeclareMathOperator{\id}{id}

\DeclareMathOperator{\dist}{d}

\DeclareMathOperator{\Fc}{\mathcal{F}}

\DeclareMathOperator{\Oc}{\mathcal{O}}

\DeclareMathOperator{\Uc}{\mathcal{U}}

\DeclareMathOperator{\Nb}{\mathbb{N}}

\DeclareMathOperator{\Rb}{\mathbb{R}}

\DeclareMathOperator{\Gsf}{\mathsf{G}}
\DeclareMathOperator{\Psf}{\mathsf{P}}

\DeclareMathOperator{\mfa}{\mathfrak{a}}

\newcommand{\abs}[1]{\left|#1\right|}

\newcommand{\norm}[1]{\left\|#1\right\|}

\newcommand{\mc}{\mathcal}

\newcommand{\ms}{\mathsf}

\DeclareMathOperator{\lox}{lox}

\begin{document}

\title{Orbital counting for relatively Anosov groups}

\author[Canary]{Richard Canary}
\address{University of Michigan}
\author[Zhang]{Tengren Zhang}
\address{National University of Singapore}
\author[Zhu]{Feng Zhu}
\address{University of Wisconsin-Madison}
\author[Zimmer]{Andrew Zimmer}
\address{University of Wisconsin-Madison}
\thanks{Canary was partially supported by grant DMS-2304636 from the National Science Foundation and a Fellowship from the Simons Foundation.
Zhang was partially supported by NUS-MOE grant A-8004148-00-00. 
Zimmer was partially supported by a Sloan Research Fellowship and grants DMS-2105580 and
DMS-2452068 from the National Science Foundation.}

\begin{abstract}
We obtain orbital counting results for relatively Anosov groups with respect to linear functionals with finite critical exponent.
Our counting results follow from an equidistribution result and rely crucially on previous equidistribution results obtained in our proof of counting
results for periods. Our results generalize earlier work of Sambarino in the setting of Anosov groups.
\end{abstract}

\date{\today}
\subjclass[2010]{}

\maketitle

\section{Introduction}

Anosov subgroups of semisimple Lie groups were introduced by Labourie \cite{labourie-invent} and are now widely recognized as the natural higher rank analogue of convex cocompact subgroups of rank one Lie groups. Relatively Anosov groups were introduced by Kapovich and Leeb \cite{KL} (see also Zhu \cite{zhu-thesis}) and are one natural higher rank analogue of geometrically finite subgroups of rank one Lie groups (another natural generalization are the extended geometrically finite groups introduced by Weisman \cite{weisman-egf}).

For any non-empty subset $\theta$ of the set of (restricted) simple roots of a semi-simple Lie group $\mathsf G$, the $\theta$-Jordan projection $\lambda_\theta$ of an element in $\mathsf G$ has image in the $\theta$-Cartan subspace $\mathfrak a_\theta$ of $\mathsf G$, and may be thought of as a vector-valued translation length function (see Section \ref{Jordan projection} for more details). If one chooses a linear functional $\phi$ on $\mathfrak a_\theta$, one may post-compose it with $\lambda_\theta$ to get a real-valued length function which we call the $\phi$-length. 
The $\theta$-Cartan projection $\kappa_\theta$ of an element in $\mathsf G$ also has image in $\mathfrak a_\theta$, and may be thought of as a vector-valued distance between a fixed basepoint and its image (see Section \ref{Cartan projection} for more details). If we choose a linear functional $\phi$ on $\mathfrak a_\theta$, we get a real valued function which we think of as the $\phi$-distance between the  basepoint and its image.  

We would naturally like  to choose $\phi$ so that the $\phi$-distance is a proper function on the group.
If $\Gamma$ is relatively $\Psf_\theta$-Anosov, this is equivalent to choosing a linear functional $\phi \in \mfa_\theta^*$ whose associated Poincar\'e series has finite critical exponent $\delta: = \delta_\Gamma^\phi$ (see \cite[Thm.\ 10.1]{CZZ4}). 
Previous work \cite[Cor.\ 10.7]{BCZZ2}, established  that the number of conjugacy classes of loxodromic elements of $\Gamma$ of $\phi$-length at most $T$ grows like $\frac{e^{\delta T}}{\delta T}$.
In this paper we determine the asymptotic behavior of the number of orbit points with $\phi$-distance at most $T$ to the basepoint.
Sambarino \cite[Thm.\ A]{sambarino-quantitative} previously established this result for $\Psf_\theta$-Anosov groups.

\begin{theorem}\label{thm:functional counting rel Anosov}
If $\Gamma\subset\mathsf G$ is a relatively $\Psf_\theta$-Anosov subgroup and $\phi \in \mfa_\theta^*$ has finite critical exponent $\delta:= \delta_\Gamma^\phi$, then there exists $C > 0$ such that 
$$
\#\{ \gamma \in \Gamma : \phi(\kappa_\theta(\gamma)) \leq R\} \sim Ce^{\delta R},
$$
i.e.\ the ratio of the two sides goes to 1 as $R \to +\infty$.
\end{theorem} 

Albuquerque \cite{albuquerque},  Link \cite{link-ps}, and Quint \cite{quint-ps}  developed a theory of Patterson--Sullivan measures for discrete subgroups of higher rank Lie groups.
For any linear functional $\phi$ on $\mathfrak a_\theta$, the $\phi$-Patterson--Sullivan measure for a discrete subgroup $\Gamma$ of $\mathsf G$ is a probability measure, supported on the limit set of $\Gamma$ in the $\theta$-flag variety $\Fc_\theta$ associated to $\mathsf G$, which is quasi-invariant with respect to the $\Gamma$-action, with Radon--Nikodym derivatives contolled by the composition of $\phi$ and the $\theta$-Iwasawa cocycle (which is a vector-valued analogue of the Busemann function in rank one, see Section \ref{Iwasawa}). Given a linear functional $\phi \in \mfa_\theta^*$, define $\bar\phi$ to be the linear functional such that $\bar\phi(\lambda_\theta(\gamma))=\phi(\lambda_\theta(\gamma^{-1}))$ for any element $\gamma\in\mathsf G$. By definition, $\phi$ and $\bar\phi$ have the same critical exponent. If $\Gamma$ is a $\Psf_\theta$-transverse subgroup, then given a $\phi$-Patterson--Sullivan measure $\mu$ and a $\bar\phi$-Patterson--Sullivan measure $\bar\mu$ (of the same dimension), one can define a BMS-measure $m=m_{\mu,\bar\mu}$ on an associated flow space $\mathsf U_\Gamma^\phi$ (see Section \ref{BMS}). 

 If $\Gamma$ is $\Psf_\theta$-relatively Anosov, then for any linear functional $\phi$ on $\mathfrak a_\theta$ with finite critical exponent $\delta^\phi_\Gamma$, there is a unique $\phi$-Patterson--Sullivan measure for $\Gamma$ of critical dimension \cite{CZZ4}, and the total mass $\norm{m}$ of $m$ is finite \cite[Thm.\ 1.1]{KO} or \cite[Thm.\ 1.4]{BCZZ2}. Our main theorem follows from the following orbital equidistribution result for well-behaved linear functionals. Below, $U_\theta(\gamma)$ is the image of the $\theta$-flag which is maximally stretched by $\gamma$ (see Section \ref{Cartan projection}) and $\mathcal{D}_x$ denotes the Dirac mass at some $x \in \Fc_\theta$

\begin{theorem} 
\label{new equidistribution}
Suppose that $\Gamma\subset\mathsf G$ is a relatively $\Psf_\theta$-Anosov subgroup and $\phi \in \mfa_\theta^*$ has finite critical exponent $\delta:= \delta_\Gamma^\phi$. If $\nu_R$ is the measure on $\Fc_\theta^2$ given by
$$
\nu_R := \delta e^{-\delta R}  \sum_{\substack{\gamma \in \Gamma \\ \phi(\kappa_\theta(\gamma)) \leq R}}\mc D_{U_\theta(\gamma^{-1})}\otimes \mc D_{U_\theta(\gamma)},$$ 
then for any continuous function on $\Fc_\theta^2$, we have
\[
 \lim_{R \to \infty} \int f\,d\nu_R= \frac{1}{\norm{m}}\int f	 \,d(\bar\mu \otimes \mu)
 \]
 where $\bar\mu$ and $\mu$ are respectively the $\bar\phi$-Patterson--Sullivan measure and $\phi$-Patterson--Sullivan measure for $\Gamma$ (of critical dimension), and $m$ is the associated BMS-measure.
\end{theorem} 

Notice that using $f\equiv 1$ in Theorem \ref{new equidistribution} gives Theorem \ref{thm:functional counting rel Anosov} with constant $C=\frac{1}{\delta \norm{m}}$. 

\medskip

The crucial tool in the proof of Theorem~\ref{new equidistribution} is an earlier equidistribution result which was developed for counting $\phi$-lengths in~\cite{BCZZ2}.
Its use follows the same outline as Sambarino's proof in the Anosov case \cite[Thm.\ A]{sambarino-quantitative}. This approach was suggested to Sambarino by Roblin.

Both of our theorems hold for the more general class of transverse subgroups and linear functionals with finite critical exponent whose associated BMS measures
has finite mass, see Section~\ref{sec:proof of the main theorem}.

\medskip\noindent
{\bf Historical remarks:}
These results and methods fit into a long history of using dynamical methods to obtain asymptotic counting results, starting with Margulis \cite{margulis} 
who established  counting, mixing  and equidistribution results for negatively curved manifolds, or more generally for Anosov flows on closed manifolds. 
Parry and Pollicott \cite{parry-pollicott} later showed how to use the symbolic dynamics given by Bowen \cite{bowen-symbolic} and the thermodynamic formalism developed by Ruelle \cite{ruelle-book} to obtain generalizations of these results.
Patterson \cite{patterson} and Sullivan \cite{sullivan-density,sullivan-hd} developed the theory of Patterson--Sullivan measures, used them to relate the critical exponent of geometrically finite hyperbolic manifolds to the Hausdorff dimension of their limit sets, and established counting results.

Roblin \cite{roblin} combined Margulis' approach with the theory of Patterson--Sullivan measures to obtain mixing, counting and equidistribution results in the CAT$(-1)$ setting. In particular, given a CAT$(-1)$ space $X$ and a discrete group $\Gamma$ of isometries of $X$ such that the associated BMS-measure is finite, we get asymptotic counting results for $\Gamma$-orbits in $X$ \cite[\S4]{roblin}. If in addition $\Gamma$ is geometrically finite, then we also get asymptotic counting for closed geodesics on $\Gamma \backslash X$ (i.e.\ closed orbits or ``periods'' of the geodesic flow on $\Gamma \backslash T^1X$) \cite[Thm.\ 5.2]{roblin}.
Roblin \cite{roblin} first obtained orbital equidistribution and used this to prove equidistribution of periods. In contrast, Sambarino \cite[Thm. A]{sambarino-quantitative} first obtained equidistribution of periods and then used this to obtain orbital equidistribution.

One might also like to have orbital counting results for the action of a discrete subgroup $\Gamma$  of $\mathsf G$ on the symmetric space $X$ associated to $\mathsf G$. Even in the Anosov case, this is only known in restricted settings. Quint \cite{quint-counting} proved orbital counting results in the symmetric space for Schottky groups in the sense of 
Benoist \cite{benoist}. Sambarino \cite{sambarino-orbital,sambarino-dichotomy}  and Edwards--Lee--Oh \cite{ELO} established symmetric space orbital counting results for Zariski dense Borel Anosov groups. Kim--Oh--Pan \cite{KOP} have announced counting results for relatively Borel Anosov  representations of geometrically finite Fuchsian groups.

\medskip\noindent
{\bf Acknowledgements:}
We thank Andr\'es Sambarino for suggesting that the technique of proof of \cite[Thm.\ A]{sambarino-quantitative} could be adapted to our setting. We also thank Dongryul Kim and Hee Oh for some helpful comments. 

This material is based upon work supported by the National Science Foundation under Grant No. DMS-2424139, while the authors were in residence at the Simons Laufer Mathematical Sciences Institute in Berkeley, California, during the Spring 2026 semester.

\section{Background}

Let $\mathsf G$ be a connected, semisimple Lie group with no compact factors and finite center, and let $\mathsf K\subset\mathsf G$ be a maximal compact Lie subgroup. Let $\mathfrak g$ denote the Lie algebra of $\mathsf G$, let $\mathfrak k\subset\mathfrak g$ be the Lie algebra of $\mathsf K$, and let $\mathfrak p\subset\mathfrak g$ be the orthogonal to $\mathfrak k$ with respect to the Killing form on $\mathfrak g$. Fix a maximal abelian subspace $\mfa \subset \mathfrak p$, and a (closed) positive Weyl chamber $\mathfrak a^+\subset\mathfrak a$. Let $\Sigma$ be the set of restricted roots on $\mfa$, and let $\Delta\subset\Sigma$ be the system of simple restricted roots associated to our choice of $\mathfrak a^+$. For each $\alpha\in\Delta$, let $\omega_\alpha$ denote the associated restricted fundamental weight.

Let $W$ be the Weyl group of $\mathfrak a$ and let $w_0\in W$ be the longest element. Then let $\iota \colon \mfa \to \mfa$ be the involution defined by $\iota(H) = -w_0 \cdot H$, and let $\iota^* \colon \mfa^*\to\mfa^*$ denote the dual of $\iota$, i.e.\ $\iota^*(\alpha):=\alpha\circ\iota$ for all $\alpha\in\mathfrak a^*$. We say that a non-empty subset $\theta\subset\Delta$ is \emph{symmetric} if $\iota^*(\theta)=\theta$. 

\subsection{Parabolic subgroups and flag manifolds} 
Let $$
\mathfrak g = \mathfrak g_0 \oplus \bigoplus_{\alpha \in \Sigma} \mathfrak g_\alpha
$$
be the restricted root space decomposition of $\mathfrak g$ associated to $\mfa$. Given a non-empty symmetric subset $\theta \subset \Delta$, let 
 $$
\mathfrak u_\theta=\mathfrak u_\theta^+:=  \bigoplus_{\alpha \in \Sigma^+_\theta} \mathfrak g_\alpha\quad\text{and}\quad\mathfrak u_\theta^- :=  \bigoplus_{\alpha \in \Sigma^+_\theta} \mathfrak g_{-\iota^*(\alpha)},
$$
where $\Sigma^+_\theta := \Sigma^+ \smallsetminus {\rm Span}( \Delta \smallsetminus \theta)$. The \emph{standard parabolic subgroup associated to $\theta$} and the \emph{standard opposite parabolic subgroup associated to $\theta$}, denoted $\Psf_\theta=\Psf_\theta^+$ and $\Psf_\theta^-$, are the normalizers in $\mathsf G$ of $\mathfrak u_\theta^+$ and $\mathfrak u_\theta^-$ respectively. 

The \emph{flag manifold associated to $\theta$}, denoted $\Fc_\theta$, is the set of conjugacy classes in $\mathsf G$ of $\Psf_\theta$. Since $\Psf_\theta$ is self-normalizing in $\mathsf G$, we have the identification
\[\Fc_\theta\cong\mathsf{G}/\mathsf{P}_\theta.\]
as $\mathsf G$-spaces.  Notice  that $\Psf_\theta^-=w_0\Psf_{\iota^*(\theta)}^+w_0^{-1}\in\Fc_{\iota^*(\theta)}$.    The set $\Fc_\theta$ is naturally a compact, smooth manifold, so we may pick a Riemannian metric on $\Fc_\theta$, and denote its associated distance function by $\dist_{\Fc_\theta}$. Since any two such Riemannian metrics on $\Fc_\theta$ are biLipschitz, for our purposes the choice of this metric is irrelevant. 

We say that a pair of flags $(x,y)\in\Fc_\theta\times\Fc_{\iota^*(\theta)}$ are \emph{transverse} if $(x, y)$ is contained in the $\mathsf G$-orbit of $(\Psf_\theta^+, \Psf_\theta^-)$. For any flag $x\in\Fc_\theta$, let $\mathcal Z_x \subset \Fc_{\iota^*(\theta)}$ denote the set of flags that are not transverse to $x$, and note that $\Fc_{\iota^*(\theta)}\smallsetminus \mathcal Z_x\subset\Fc_{\iota^*(\theta)}$ is a dense open subset. 
We also denote
\[\Fc_\theta^{(2)}:=\{(x,y)\in\Fc_\theta\times\Fc_{\iota^*(\theta)}:x\text{ and }y\text{ are transverse}\}.\]

For any non-empty subset $\theta'\subset\theta$, observe that $\Psf_\theta\subset\Psf_{\theta'}$. Thus, the \emph{forgetful map} $\pi_{\theta,\theta'}\colon \Fc_\theta\to\Fc_{\theta'}$ that sends $g\Psf_\theta\mapsto g\Psf_{\theta'}$ is well-defined. Notice that if $(x,y)\in\Fc_\theta^{(2)}$, then $(\pi_{\theta,\theta'}(x),\pi_{\iota^*(\theta),\iota^*(\theta')}(y))\in\Fc_{\theta'}^{(2)}$.

\subsection{Lie group decompositions}
We will now recall the $\mathsf{KAK}$-decomposition, the Jordan decomposition, and the Iwasawa decomposition of $\mathsf G$. These will allow us to define the Cartan projection, the Jordan projection, and the Iwasawa cocycle.

\subsubsection{Cartan projection} \label{Cartan projection}
By the $\mathsf{KAK}$-decomposition theorem, every $g\in\mathsf G$ can be written as 
$$
g = m e^{\kappa(g)} \ell
$$
for some $m, \ell \in \mathsf K$ and a unique $\kappa(g)\in\mfa^+$. Using this, one can define a \emph{Cartan projection}
\[\kappa \colon \mathsf G \rightarrow \mfa^+.\] 
Note that $\iota( \kappa(g)) = \kappa(g^{-1})$ for all $g \in \mathsf G$. 

Given a non-empty subset $\theta \subset \Delta$, the \emph{$\theta$-Cartan subspace} is 
$$
\mfa_\theta := \{ X \in \mfa : \alpha(X) = 0 \text{ for all } \alpha \in \Delta \smallsetminus \theta\}
$$
and the \emph{$\theta$-positive Weyl chamber} is 
\[\mfa_\theta^+:=\mfa_\theta\cap\mfa^+.\]
One can verify that $\{ \omega_\alpha|_{\mfa_\theta} : \alpha \in \theta\}$ is a basis of $\mfa_\theta^*$. Hence, we may identify 
\begin{align}\label{usual identification}
\mfa_\theta^*={\rm Span}\{ \omega_\alpha : \alpha \in \theta\}\subset\mfa^*.
\end{align}
Let $p_\theta \colon \mfa \rightarrow \mfa_\theta$ be the projection with the defining property that $\omega_\alpha(X)=\omega_\alpha(p_\theta(X))$ for all $\alpha\in\theta$ and $X\in\mfa$. The \emph{$\theta$-Cartan projection} is
\[\kappa_\theta:=p_\theta\circ\kappa \colon \mathsf G\to\mfa_\theta^+.\]
Note that $\phi(\kappa_\theta(g))=\phi(\kappa(g))$ for all $\phi \in \mfa_\theta^*$ and $g\in\mathsf G$, where $\phi$ on the right uses the identification \eqref{usual identification}.

For each $g\in\mathsf G$, choose a $\mathsf{KAK}$-decomposition $g=me^{\kappa(g)}\ell$, and define $U_\theta(g):=m\mathsf P_\theta\in\Fc_\theta$. In general, $U_\theta(g)$ depends on the choice of the $\mathsf{KAK}$-decomposition of $g$. However, if $\alpha(\kappa(g))>0$ (equivalently, $\alpha(\kappa_\theta(g))>0$) for all $\alpha\in\theta$, then $U_\theta(g)$ does not depend on this choice.

The following proposition characterizes when a sequence $(g_n)$ in $\mathsf G$ acts on $\Fc_\theta$ and $\Fc_{\iota^*(\theta)}$ with some north-south dynamics, see \cite[Prop.\ A.1]{CZZ3} for a proof.

\begin{proposition}\label{prop:characterizing convergence in general} Suppose $(x^+,x^-)\in \Fc_\theta\times\Fc_{\iota^*(\theta)}$ and $(g_n)$ is a sequence in $\mathsf G$. The following are equivalent:  
\begin{enumerate}
\item $U_\theta(g_n) \rightarrow x^+$, $U_{\iota^*(\theta)}(g_n^{-1}) \rightarrow x^-$ and $\lim_{n \rightarrow \infty} \alpha(\kappa(g_n)) = \infty$ for every $\alpha \in \theta$,
\item $g_n|_{\Fc_\theta\smallsetminus \mathcal Z_{x^-}} \colon \Fc_\theta\smallsetminus \mathcal Z_{x^-}\to\Fc_\theta$ converges uniformly on compact sets to the constant map whose image is $x^+$,
\item $g_n^{-1}|_{\Fc_{\iota^*(\theta)}\smallsetminus \mathcal Z_{x^+}} \colon \Fc_{\iota^*(\theta)}\smallsetminus \mathcal Z_{x^+}\to\Fc_{\iota^*(\theta)}$ converges uniformly on compact sets to the constant map whose image is $x^-$,
\end{enumerate}
\end{proposition}

As a consequence, the maps $U_\theta$ have the following equivariance property. 

\begin{corollary}\label{cor:equivariance of U_theta} If $(g_n)$ is a sequence in $\Gsf$ with $\lim_{n \rightarrow \infty} \alpha(\kappa(g_n)) = \infty$ for every $\alpha \in \theta$ and $(a_n)$ and  $(b_n)$ are bounded sequences in $\Gsf$, then 
$$
\lim_{n \rightarrow \infty} \dist_{\Fc_\theta}(a_nU_\theta(g_n), U_\theta(a_ng_nb_n)) = 0.
$$
\end{corollary} 

\begin{proof} It suffices to consider the case where $a_n \rightarrow a$, $b_n \rightarrow b$, $U_\theta(g_n) \rightarrow x^+$, and $U_{\iota^*(\theta)}(g_n^{-1}) \rightarrow x^-$. Then applying Proposition~\ref{prop:characterizing convergence in general}  to $(g_n)$ we have
$$
a_n g_nb_n x \rightarrow a x^+
$$
for all $x \in \Fc_\theta\smallsetminus \mathcal Z_{b^{-1}x^-}$ and the convergence is uniform on compact subsets. Hence  Proposition~\ref{prop:characterizing convergence in general} implies that $U_\theta(a_ng_nb_n) \rightarrow ax^+$. 
\end{proof}

\subsubsection{Jordan projection}\label{Jordan projection}
Let $\mathsf U:=\exp(\mathfrak u_\Delta)$ and $\mathsf A:=\exp(\mathfrak a)$.
By the Jordan decomposition theorem, each $g\in\mathsf G$ can be written uniquely as a commuting product $g=g_eg_ug_h$, 
where $g_e$ is conjugate to an element in $\mathsf K$, $g_u$ is conjugate to an element in $\mathsf U$, and $g_h$ is conjugate to an element in $\mathsf A$. 
The \emph{Jordan projection}
\[\lambda \colon \mathsf G\to\mathfrak a^+\]
is the map so that $\lambda(g)$ is the unique element of $\mathfrak a^+$ so that  $e^{\lambda(g)}$ is conjugate to $g_h$. 
Then for any non-empty subset $\theta\subset\Delta$, we may define the \emph{$\theta$-Jordan projection}
\[\lambda_\theta:=p_\theta\circ\lambda \colon \mathsf G\to\mathfrak a_\theta^+.\]

We say that $g\in\mathsf G$ is \emph{$\theta$-loxodromic} if $\alpha(\lambda(g))>0$ (equivalently, $\alpha(\lambda_\theta(g))>0$) for all $\alpha\in\theta$. Observe that if $g\in\mathsf G$ is $\theta$-loxodromic, then $g$ has an attracting fixed point $g^+\in\Fc_\theta$ and a repelling fixed point $g^-\in\Fc_{\iota^*(\theta)}$, both of which are necessarily unique.

\subsubsection{Iwasawa cocycle}\label{Iwasawa} By the Iwasawa decomposition theorem, every $g\in\mathsf G$ can be written uniquely as $g=kau$, where $k\in\mathsf K$, $a\in\mathsf A$, and $u\in\mathsf U$. One can then define the \emph{Iwasawa cocycle}
\[
B \colon \mathsf G \times \Fc_\Delta \rightarrow \mfa
\] 
with the defining property that $gm = k e^{B(g,x)} u$ for all $(g,x)\in \mathsf G \times \Fc_\Delta$, where $m\in\mathsf K$ is an element such that $x = m\Psf_\Delta$. 

For any non-empty subset $\theta\subset\Delta$, let $\pi_\theta:=\pi_{\Delta,\theta}:\Fc_\Delta\to\Fc_\theta$ be the forgetful map. The \emph{$\theta$-Iwasawa cocycle} is the map 
\[B_\theta \colon \mathsf G \times \Fc_\theta \rightarrow \mfa_\theta\] 
defined by $B_\theta(g,x) = p_\theta( B(g,x') )$ for some (all) $x'\in\pi_\theta^{-1}(x)$.
By Quint \cite[Lem.\ 6.1 and 6.2]{quint-ps}, this is a well-defined group cocycle, that is 
$$
B_\theta(gh, x) = B_\theta(g, hx) + B_\theta(h, x)
$$
for all $g,h \in \mathsf G$ and $x \in \Fc_\theta$.

\subsection{Tits representation and the Gromov product}
In the case when $\mathsf G=\SL(d,\Rb)$ and $\mathsf K=\SO(d)$, we may choose $\mathfrak a$ to be the vector space of diagonal, traceless, real-valued, $d\times d$ matrices and $\mathfrak a^+$ to be the set of diagonal matrices in $\mathfrak a$ where the diagonal entries are weakly decreasing down the diagonal. Then $\Delta=\{\alpha_1,\dots,\alpha_{d-1}\}$, where for each $k\in\{1,\dots,d-1\}$, the restricted simple root
\[\alpha_k \colon \mathfrak a\to\Rb\] 
is given by $\alpha_k({\rm diag}(A_1,\dots,A_d))=A_k-A_{k+1}$, and the restricted fundamental weight 
\[\omega_k:=\omega_{\alpha_k} \colon \mathfrak a\to\Rb\]
is given by $\omega_k({\rm diag}(A_1,\dots,A_d))=\sum_{i=1}^kA_i$. If $(i_1,\dots,i_k)$ is an increasing tuple of integers weakly between $1$ and $d-1$, we will also denote $\Fc_{i_1,\dots,i_k}(\Rb^d):=\Fc_{\alpha_1,\dots,\alpha_k}$.

In this case, we may describe the Cartan projection, Jordan projection, and Iwasawa cocycle using linear algebra as well. The Cartan projection $\kappa$ is given by
\[\kappa(g)={\rm diag}(\log \sigma_1(g), \dots, \log \sigma_d(g)),\]
where $\sigma_1(g)\ge\dots\ge\sigma_d(g)$ denote the singular values of $g$ in the standard inner product on $\Rb^d$. The Jordan projection $\lambda$ is given by
\[\lambda(g)={\rm diag}(\log \lambda_1(g), \dots, \log \lambda_d(g)),\]
where $\lambda_1(g)\ge\dots\ge\lambda_d(g)$ denote the moduli of the (generalized) eigenvalues of $g$. Finally, the Iwasawa cocycle satisfies the defining property that for all $k\in\{1,\dots,d-1\}$,
\[\omega_k(B(g,x))=\log\frac{\norm{g(v_1)\wedge\dots\wedge g(v_k)}_k}{\norm{v_1,\dots,v_k}_k}\]
where $v_1,\dots,v_k\in\Rb^d$ is a basis for $x^k$  and $\norm{\cdot}_k$ is the norm on $\wedge^k\Rb^d$ induced by the standard inner product on $\Rb^d$.

Returning to the situation for general $\mathsf G$, we have the following well-known proposition, which gives a convenient way to linearize $\mathsf G$. 

\begin{proposition}\label{prop:reduction to the linear case} For any $\alpha\in\Delta$, let $\theta_\alpha:=\{\alpha,\iota^*(\alpha)\}$. There exists an irreducible linear representation $\Phi_\alpha \colon \mathsf G \rightarrow \mathsf{SL}(d_\alpha,\Rb)$ that sends $\mathsf K$ to $\mathsf{SO}(d_\alpha)$, a $\Phi_\alpha$-equivariant smooth embedding
$$
\xi_\alpha \colon \Fc_{\theta_\alpha} \to \Fc_{1,d_\alpha-1}(\Rb^{d_\alpha}),
$$
and a positive integer $N_\alpha$ such that: 
\begin{enumerate}
 \item $x_1, x_2 \in \Fc_{\theta_\alpha}$ are transverse if and only if $\xi_\alpha(x_1),\xi_\alpha(x_2)\in\Fc_{1,d_\alpha-1}(\Rb^{d_\alpha})$ are transverse.
 \item For all $g \in \mathsf G$,
\[\omega_1\big(\kappa(\Phi_\alpha(g))\big) =N_\alpha\omega_\alpha\big(\kappa(g)\big),\quad\omega_1\big(\lambda(\Phi_\alpha(g))\big)  =N_\alpha\omega_\alpha\big(\lambda(g)\big),\]
\[\omega_{d-1}(\kappa(\Phi_\alpha(g))) =N_\alpha \omega_{\iota^*(\alpha)}(\kappa(g))\quad\text{and}\quad\omega_{d-1}(\lambda(\Phi_\alpha(g)))  = N_\alpha \omega_{\iota^*(\alpha)}(\lambda(g)).\]
\item For all $(g,x)\in\mathsf G\times\Fc_{\theta_\alpha}$, 
\[\omega_1\big(B_{1,d_\alpha-1}(\Phi_\alpha(g),\xi_\alpha(x))\big) = N_\alpha \omega_\alpha\big(B_{\theta_\alpha}(g,x)\big).\]
\item For all $g\in\mathsf G$, $\alpha_1(\kappa(\Phi_\alpha(g))) = \alpha( \kappa(g))$ and $\alpha_1(\lambda(\Phi_\alpha(g))) = \alpha( \lambda(g))$. 
\item If $g\in\mathsf G$ satisfies $\alpha(\kappa(g))>0$ and $(\iota^*\alpha)(\kappa(g)) > 0$, then
$$
\xi_\alpha( U_{\theta_\alpha}(g)) = U_{1,d_\alpha-1}(\Phi_\alpha(g)). 
$$
\end{enumerate}
\end{proposition}

\begin{proof} The existence of such a representation and properties (1), (2), and (4) follow from the discussion in~\cite[Sec.\ 3]{GGKW}. More precisely, they are special cases of \cite[Prop.\ 3.3(b)]{GGKW}, \cite[Prop.\ 2.3]{GGKW}, and \cite[Prop.\ 3.7(3)]{GGKW} (also see \cite[Lem.\ 2.13]{smilga}) respectively. 
Property (3) follows from~\cite[Lem.\ 6.4]{quint-ps} and Property (5) follows from~\cite[Lem.\ B.5]{CZZ3}. 
\end{proof}

We refer to any representation $\Phi_\alpha$ satisfying the conclusions of Proposition \ref{prop:reduction to the linear case} as a \emph{Tits representation associated to $\alpha$}.

\subsubsection{Gromov product}
Using the Tits representations, we will now define the Gromov product. Let $\theta\subset\Delta$ be non-empty and symmetric. 
For all $\alpha\in\theta$, let $\theta_\alpha:=\{\alpha,\iota^*(\alpha)\}$, and let $\pi_{\theta_\alpha}:=\pi_{\theta,\theta_\alpha} \colon \Fc_\theta\to\Fc_{\theta_\alpha}$ 
denote the forgetful map. Moreover, let $\Phi_\alpha$ and $\xi_\alpha$ be a Tits representation associated to  $\alpha$ and its corresponding limit map $\xi_\alpha$.

 Following Sambarino \cite[pp.\ 484]{sambarino-quantitative}, the \emph{$\theta$-Gromov product}
\[G_\theta\colon \Fc_\theta^{(2)}\to\mathfrak a_\theta^+\]
is the map with the defining property that for all $\alpha\in\theta$ and all $(x_1,x_2)\in\Fc_\theta^{(2)}$, we have
\[\omega_\alpha(G_\theta(x_1,x_2))=\frac{1}{N_\alpha}\log\frac{\norm{f}_\alpha^*\norm{v}_\alpha}{\abs{f(v)}},\]
where $f\in(\Rb^{d_\alpha})^*$ is some (any) covector whose kernel is $\xi_\alpha(\pi_{\theta_\alpha}(x_1))^{d_\alpha-1}$, $v\in\Rb^{d_\alpha}$ is some (any) non-zero vector in $\xi_\alpha(\pi_{\theta_\alpha}(x_2))^1$, and $\norm{\cdot}_\alpha$ and $\norm{\cdot}_\alpha^*$ respectively denote the norms on $\Rb^{d_\alpha}$  and $(\Rb^{d_\alpha})^*$ induced by the standard inner product on $\Rb^{d_\alpha}$. By Proposition \ref{prop:reduction to the linear case} (1), this quantity is a real number.

 Also, for all $g\in\Gsf$ and all $(x_1,x_2)\in\Fc_\theta^{(2)}$, we have
\[G_\theta(g(x_1),g(x_2))-G_\theta(x_1,x_2)=\iota(B_\theta(g,x_1))+B_\theta(g,x_2)\]
(see \cite[Lem.\ 7.9]{sambarino-quantitative}). Furthermore, $G_\theta$ is continuous and proper. Also, if $\alpha\in\theta'\subset\theta$, then 
\[\omega_\alpha(G_\theta(x_1,x_2))=\omega_\alpha(G_{\theta'}(\pi_{\theta,\theta'}(x_1),\pi_{\theta,\theta'}(x_2))),\]
where $\pi_{\theta,\theta'}\colon \Fc_\theta\to\Fc_{\theta'}$ is the forgetful map.

\begin{remark}
The Gromov product that we use here is negative of the one used by Sambarino (cf.\ \cite[pp.\ 484]{sambarino-quantitative}).
\end{remark}

\subsubsection{Linearization via Tits representations}
Often, the Tits representations allow us to reduce the proof of statements about $\Gsf$ to their linear algebraic counterparts. The following corollary is an example of this that will be useful later.

\begin{corollary}\label{subadditive}
For any $\alpha\in\theta$ and $g,h_1, h_2\in\Gsf$, 
\[-\omega_{\iota^*(\alpha)}(\kappa(h_1)+\kappa(h_2))\le\omega_\alpha(\kappa(h_1g h_2)-\kappa(g))\le\omega_\alpha(\kappa(h_1)+\kappa(h_2)).\]
\end{corollary}

\begin{proof}
By Proposition \ref{prop:reduction to the linear case} (2), the required inequality is equivalent to
\[\sigma_d(\Phi_\alpha(h_1))\sigma_d(\Phi_\alpha(h_2))\le\frac{\sigma_1(\Phi_\alpha(h_1gh_2))}{\sigma_1(\Phi_\alpha(g))}\le\sigma_1(\Phi_\alpha(h_1))\sigma_1(\Phi_\alpha(h_2)),\]
which is a consequence of the ``minimax" definition of singular values.
\end{proof}

\subsection{Divergent, transverse, Anosov, and relatively Anosov subgroups} 
A subgroup $\Gamma\subset\mathsf G$ is \emph{$\Psf_\theta$-divergent} if $\alpha(\kappa(\gamma_n))\to\infty$ for any $\alpha\in\theta$ and any sequence $(\gamma_n)$ in $\Gamma$ of pairwise distinct elements. Note that all $\Psf_\theta$-divergent subgroups are discrete. If $\Gamma$ is $\Psf_\theta$-divergent, the \emph{$\theta$-limit set} $\Lambda_\theta(\Gamma)$ of $\Gamma$ is the set of limit points in $\Fc_\theta$ of $\{U_\theta(\gamma):\gamma\in\Gamma\}$. One can verify that $\Lambda_\theta(\Gamma)$ is a closed, $\Gamma$-invariant subset of $\Fc_\theta$. We say that $\Gamma$ is {\em non-elementary} if $\Lambda_\theta(\Gamma)$ is infinite. A $\Psf_\theta$-divergent subgroup $\Gamma\subset\mathsf G$ is \emph{$\Psf_\theta$-transverse} if $\Lambda_\theta(\Gamma)$ is a transverse subset of $\Fc_\theta$, i.e.\ distinct pairs of flags in $\Lambda_\theta(\Gamma)$ are transverse. 

The action of a discrete group $\Gamma_0$ on a compact metric space $X$ is said to be a (discrete) {\em convergence group action} if for any sequence $(\gamma_n)$ of distinct elements in $\Gamma_0$, there are points $x,y\in X$ and a subsequence $(\gamma_{n_j})$ such that  the sequence of maps $\gamma_{nj}|_{X\smallsetminus\{y\}} \colon X\smallsetminus\{y\}\to X$ converges uniformly on compact subsets to the constant map whose image is $x$.

\begin{proposition}[{\cite[Sec.\ 5.1]{KLP1}, \cite[Prop.\ 3.3]{CZZ2}}]\label{prop: convergence group}
If $\Gamma$ is $\Psf_\theta$-transverse, then $\Gamma$ acts on $\Lambda_\theta(\Gamma)$ as a convergence group. 
In particular, if $\Gamma$ is non-elementary, then $\Gamma$ acts on $\Lambda_\theta(\Gamma)$ minimally, and $\Lambda_\theta(\Gamma)$ is perfect. 
\end{proposition}

If a group $\Gamma_0$ acts on a metric space $X$ as a convergence group, we say that a point $x\in X$  is a 
\begin{itemize}
\item {\em conical limit point} if  there exist distinct $a,b\in X$ and a sequence
$(\gamma_n)$ in $\Gamma_0$ so that $\gamma_n(x)\to a$ and $\gamma_n(y)\to b$ for all $y\in X\smallsetminus\{x\}$.  
\item {\em bounded parabolic point} if its stabilizer in $\Gamma$ acts properly and cocompactly on $X\smallsetminus\{x\}$.
\end{itemize}

A $\Psf_\theta$-transverse subgroup $\Gamma\subset\mathsf G$ is \emph{$\Psf_\theta$-Anosov} (respectively, \emph{$\Psf_\theta$-relatively Anosov}) if every point in $\Lambda_\theta(\Gamma)$ is a conical limit point (respectively, either a conical limit point or a bounded parabolic point) for the $\Gamma$-action. 

This is not the original definition of an Anosov group (given in \cite{labourie-invent, GW}) but can be deduced from the characterizations in~\cite{GGKW,KLP1}. Further, the equivalence of this definition and the original one follows directly from ~\cite[Thm.\ 1.1]{tsouvalas-gaps} and Proposition~\ref{prop:characterizing convergence in general}. For a discussion of the various definitions of relatively Anosov groups, see \cite[Sec.\ 4]{ZZ1}.

Observe that if a subgroup is $\Psf_\theta$-divergent (respectively, $\Psf_\theta$-transverse, $\Psf_\theta$-Anosov, $\Psf_\theta$-relatively Anosov) then it is $\Psf_{\iota^*(\theta)}$-divergent (respectively, $\Psf_{\iota^*(\theta)}$-transverse, $\Psf_{\iota^*(\theta)}$-Anosov, $\Psf_{\iota^*(\theta)}$-relatively Anosov), and is also $\Psf_{\theta'}$-divergent (respectively, $\Psf_{\theta'}$-transverse, $\Psf_{\theta'}$-Anosov, $\Psf_{\theta'}$-relatively Anosov) for all non-empty $\theta'\subset\theta$.

\subsection{Patterson--Sullivan measures for divergent groups} \label{sec:PS measures for divergent grps}
Let $\theta \subset \Delta$ be a non-empty symmetric subset, let $\Gamma\subset\mathsf G$ be a $\Psf_\theta$-divergent group, and let $\phi\in\mathfrak a_\theta^*$. A \emph{$\phi$-Patterson--Sullivan measure for $\Gamma$ of dimension $\beta$} is a probability measure $\mu$ supported on $\Lambda_\theta(\Gamma)$ such that for any $\gamma \in \Gamma$, the measures $\mu$ and $\gamma_*\mu$ are absolutely continuous and 
$$
\frac{d\gamma_*\mu}{d\mu}(x)=e^{-\beta\phi(B_\theta(\gamma^{-1},x))}
$$
for $\mu$-almost every $x\in\Lambda_\theta(\Gamma)$. 
Note that as finite Borel measures on a metric space, $\mu$ and $\gamma_*\mu$ are necessarily regular. 

For any $\phi\in\mfa_\theta^*$, the \emph{$\phi$-Poincar\'e series for $\Gamma$} is
$$Q_\Gamma^\phi(s):=\sum_{\gamma\in\Gamma} e^{-s\phi(\kappa_\theta(\gamma))}.$$
The \emph{$\phi$-critical exponent of $\Gamma$}, denoted $\delta^\phi_\Gamma$, is the critical exponent of $Q_\Gamma^\phi(s)$, i.e. 
\[
\delta^\phi_\Gamma:=\inf\{s > 0:Q_\Gamma^\phi(s)<+\infty\}.
\]
For any discrete subgroup $\Gamma\subset\mathsf G$, the \emph{$\theta$-Benoist limit cone} of $\Gamma$ is the set of points in $\mathfrak a_\theta^+$ that are limits of sequences of the form $(t_n\kappa_\theta(\gamma_n))$, where $(\gamma_n)$ is a distinct sequence in $\Gamma$ and $(t_n)$ is a decreasing sequence of real numbers that converge to $0$.

\begin{proposition}\label{prop: Patterson-Sullivan}
If $\theta \subset \Delta$ is non-empty and symmetric, $\Gamma\subset\mathsf G$ is $\Psf_\theta$-divergent, $\phi\in \mfa^*_\theta$ and $\delta^\phi_\Gamma < +\infty$, then there is a $\phi$-Patterson--Sullivan measure $\mu$ for $\Gamma$ of dimension $\delta^\phi_\Gamma$.
\end{proposition}

If $\Gamma\subset\mathsf G$ is $\Psf_\theta$-transverse, and $\phi\in\mathfrak a_\theta^*$ satisfies $\delta^\phi_\Gamma<+\infty$ and 
$Q^\phi_\Gamma(\delta^\phi_\Gamma)=+\infty$, we say that $(\Gamma,\phi)$ is \emph{of divergence type}. The Hopf--Tsuji--Sullivan dichotomy (see \cite[Thm.\ 1.3]{BCZZ1} or \cite[Thm. 1.9]{KOW}) implies that if $(\Gamma,\phi)$ is of divergence type, then there is a unique $\phi$-Patterson--Sullivan measure for $\Gamma$ of dimension $\delta^\phi_\Gamma$, 
which we denote by $\mu^\phi_\Gamma$. 

When $\Gamma$ is $\Psf_\theta$-relatively Anosov, $\phi\in\mathfrak a_\theta^*$, and $\delta^\phi_\Gamma < +\infty$,  then $(\Gamma,\phi)$ is of divergence type, by \cite[Thm.\ 1.1]{CZZ4}.

\subsection{BMS measures}\label{BMS}
Let $\Gamma\subset\mathsf G$ be a $\Psf_\theta$-transverse subgroup, and let $\phi\in\mathfrak a_\theta^*$ be such that $(\Gamma,\phi)$ is of  divergence type.
Set $\bar\phi:=\iota^*(\phi)$. Since $\phi(\kappa_\theta(\gamma))=\bar \phi(\kappa_\theta(\gamma^{-1}))$, note that $\delta^\phi_\Gamma=\delta^{\bar\phi}_\Gamma$ and the pair $(\Gamma,\bar\phi)$ is also of divergence type.
Denote $\delta:=\delta^\phi_\Gamma=\delta^{\bar\phi}_\Gamma$, $\mu:=\mu^\phi_\Gamma$, and $\bar\mu:=\mu^{\bar\phi}_\Gamma$. 

Let $\Lambda_\theta(\Gamma)^{(2)}$ denote the set of distinct pairs of points in $\Lambda_\theta(\Gamma)$. Observe that the topological space 
\[\widetilde{\mathsf U}_\Gamma:=\Lambda_\theta(\Gamma)^{(2)}\times \Rb\] 
admits a natural continuous flow $\psi_t$ given by $\psi_t(x,y,s)=(x,y,s+t)$. 
Using $\phi$, we may define a $\Gamma$-action on $\widetilde{\mathsf U}_\Gamma$ as follows:
\[
\gamma\cdot(x,y,s)=\big(\gamma\cdot x,\gamma\cdot y,s+\phi(B_\theta(\gamma,y))\big).
\]
Notice that this $\Gamma$-action on $\widetilde{\mathsf U}_\Gamma$ commutes with the flow $\psi_t$, and is properly discontinuous \cite[Thm.\ 9.1]{KOW}. Thus, $\psi_t$ descends to a flow, also denoted by $\psi_t$, on the Hausdorff space $\mathsf U_\Gamma^\phi:=\Gamma \backslash \widetilde{\mathsf U}_\Gamma$. 

Using $\mu$ and $\bar\mu$, we may define a $\psi_t$-invariant measure on $\mathsf U_\Gamma^\phi$ as follows. Let $\widetilde m$ be the measure on $\widetilde{\mathsf U}_\Gamma$ defined by
\[d\widetilde m(x,y,s):=e^{\delta^\phi_\Gamma\phi(G(x,y))}d\bar\mu(x)d\mu(y)dL(s),\]
where $L$ denotes the usual Lebesgue measure on $\Rb$. One can verify that $\widetilde m$ is both $\psi_t$-invariant and invariant under the $\Gamma$-action on $\widetilde{\mathsf U}_\Gamma$ defined above. As such, it descends to a $\psi_t$-invariant measure $m=m_\Gamma^\phi$ on $\mathsf U_\Gamma^\phi$, called the \emph{$\phi$-Bowen-Margulis-Sullivan (BMS) measure} for $\Gamma$.

Let $\Gamma_{\rm lox}$ denote the set of $\theta$-loxodromic elements in $\Gamma$. When the $\phi$-BMS measure for $\Gamma$ is finite, the following equidistribution result applies.

 \begin{theorem}
 \label{old equidistribution}
Suppose that $\Gamma\subset \mathsf G$ is a $\Psf_\theta$-transverse subgroup, $\phi\in \mathfrak a_\theta^*$,  $(\Gamma,\phi)$ is of divergence type, and the total mass $\norm{m}$ of the $\phi$-BMS measure $m$ for $\Gamma$ is finite. Let $\delta: = \delta_\Gamma^\phi$. Then 
 \[
 \lim_{T \rightarrow \infty} \delta e^{-\delta  T} \sum_{\substack{\gamma \in \Gamma_{\lox}\\ \phi(\lambda_\theta(\gamma))\leq T}}\mc D_{\gamma^-}\otimes \mc D_{\gamma^+} =
  \frac{1}{\norm{m}} e^{\delta  \phi(G_\theta(x,y))} \bar\mu(x)\otimes \mu(y)
 \]
  in the dual to compactly-supported continuous functions on $\mathcal F_\theta^{(2)}$.
\end{theorem}

\begin{proof} By  \cite[Prop.\ 10.4 and Thm.\ 1.3]{BCZZ2} or \cite[Thm\ 1.1]{KO}, the flow $\psi^t$ is mixing with respect to $m$, and hence the theorem follows from \cite[Thm.\ 6.1]{BCZZ2}. \end{proof}

\section{The proof of Theorem \ref{new equidistribution}}\label{sec:proof of the main theorem}

Let $\theta\subset\Delta$ be non-empty and symmetric. In this section we establish the following generalization of Theorem \ref{new equidistribution}.

\begin{theorem}\label{new equidistribution gen ver}  Suppose that $\Gamma\subset \mathsf G$ is a $\Psf_\theta$-transverse subgroup, $\phi\in \mathfrak a_\theta^*$,  $(\Gamma,\phi)$ is of divergence type, and the total mass $\norm{m}$ of the $\phi$-BMS measure $m$ for $\Gamma$ is finite.  Let $\delta : =\delta^\phi_\Gamma$. 

If $\nu_R$ is the measure on $\Fc_\theta^2$ given by
$$
\nu_R := \delta e^{-\delta R}  \sum_{\substack{\gamma \in \Gamma \\ \phi(\kappa(\gamma)) \leq R}}\mc D_{U_\theta(\gamma^{-1})}\otimes \mc D_{U_\theta(\gamma)},$$ 
then for any continuous function $f \colon \Fc_\theta^2 \rightarrow \Rb$ we have
\[
 \lim_{R \to \infty} \int f\,d\nu_R= \frac{1}{\norm{m}}\int f	 \,d(\bar\mu \otimes \mu).
 \]
In particular, 
$$
\#\{ \gamma \in \Gamma : \phi(\kappa(\gamma)) \leq R\} \sim \frac{1}{\delta \norm{m}}e^{ \delta R}.
$$
\end{theorem} 

Notice that the ``in particular'' part follows from the main assertion by taking $f \equiv 1$. 

 When $\Gamma\subset \mathsf G$ is a $\Psf_\theta$-relatively Anosov subgroup, $\phi\in \mathfrak a_\theta^*$, and $\delta^\phi_\Gamma < +\infty$, then $(\Gamma,\phi)$ is of divergence type by~\cite{CZZ4}. The $\phi$-BMS measure $m$ for $\Gamma$ is finite by either \cite[Thm. 1.1]{KO} or \cite[Thm.\ 8.1, Prop.\ 10.3]{BCZZ2} (and~\cite[Cor.\ 7.2]{CZZ4}). Hence Theorem~\ref{new equidistribution gen ver} does indeed generalize Theorems~\ref{thm:functional counting rel Anosov} and  \ref{new equidistribution} from the introduction.

The proof  has three main steps. 
First, we prove that one can estimate the difference between the $\theta$-Jordan projection and the $\theta$-Cartan projection of the elements using the $\theta$-Gromov product, see Lemmas \ref{prop:length versus magnitude} and \ref{lem:length versus magnitude 2}. 
Using this and Theorem~\ref{old equidistribution}, we then prove a weaker version of Theorem \ref{new equidistribution}, where we further assume that the support of $f$ is a compact subset of $\Fc_\theta^{(2)}$, see Proposition \ref{off the diagonal}. 
Finally, we prove that the $\nu_R$-measure  can be made arbitrarily small on small open neighborhoods of $\Fc_\theta^2\smallsetminus\Fc_\theta^{(2)}$, see Proposition \ref{on the diagonal}. Together, these three lemmas imply Theorem \ref{new equidistribution}, see Section \ref{sec:finishing proof}.

\subsection{Cartan projections, Jordan projections, and the Gromov product}

The next two lemmas control the difference between the $\theta$-Jordan projection and the $\theta$-Cartan projection using the $\theta$-Gromov product.

\begin{lemma}\label{prop:length versus magnitude} 
For any $\epsilon > 0$ and compact subset $K \subset \Fc_\theta^{(2)}$ there exists $R > 0$ such that: if $g \in \Gsf$, 
$(U_\theta(g^{-1}), U_\theta(g)) \in K$, and $\min_{\alpha \in \theta} \alpha(\kappa(g)) \geq R$, then
\begin{enumerate}
\item $g$ is $\theta$-loxodromic, and $\dist_{\Fc_\theta}(U_\theta(g^{\pm 1}), g^\pm) < \epsilon$,
\item $\left|\omega_\alpha\Big(\kappa_\theta(g) - \lambda_\theta(g) - G_\theta( U_\theta(g^{-1}),U_\theta(g)\Big)\right| < \epsilon$ for all $\alpha\in\theta$.
\end{enumerate} 
\end{lemma}

\begin{proof}
We first prove the lemma in the special case where ${\ms G} = \SL(d,\Rb)$, and $\theta=\{\alpha_1,\alpha_{d-1}\}$. Then, we will deduce the full strength of the proposition from the special case by applying the representations given by Proposition \ref{prop:reduction to the linear case}.

Suppose that the lemma fails in the special case. Then there exist $\epsilon>0$, a compact subset $K \subset \Fc_\theta^{(2)}$ and a  sequence $(g_n)$ in $\Gsf$
such that for each $n\in\Nb$, 
\[
\min_{\alpha \in \theta} \alpha(\kappa(g_n)) \geq n \quad \text{and} \quad (U_\theta(g_n^{-1}),U_\theta(g_n)) \in K,
\]
 and at least one of the conclusions of the lemma fails. Passing to a subsequence, we may assume that 
\[
U_{\theta}(g_n)\to x_+\quad\text{and}\quad U_\theta(g_n^{-1})\to x_-
\]
for some transverse $x_+,x_-\in\Fc_\theta$.  
 
For any $g\in\End(\Rb^d)$, let $\norm{g}$ denote the largest singular value of $g$. Passing to a further subsequence we can suppose that $ \frac{g_n}{\|g_n\|} \rightarrow T$ in $\End(\Rb^d)$. Since $\alpha_1(\kappa(g_n))\to\infty$, 
$$
{\rm im} \, T = x_+^1 \quad \text{and} \quad \ker T = x_-^{d-1}. 
$$
Fix unit vectors $v \in \Rb^d$ and $f \in \Rb^{d*}$ with $v \in x_+^1$ and $\ker f = x_-^{d-1}$. Since $\norm{T} =1$, after possibly replacing $v$ with $-v$ we have 
$
T ( \cdot )=  f(\cdot ) v.
$
Since $x_+$ and $ x_-$ are transverse, $f(v) \neq 0$. Then 
$$
\lambda_1(T) =\norm{T(v)} = \abs{f(v)} \quad \text{and} \quad \lambda_2(T) = \cdots = \lambda_d(T) = 0.
$$
Then, since eigenvalues are continuous functions on the space of matrices, 
$g_n = \norm{g_n} \frac{g_n}{\|g_n\|}$ is proximal when $n$ is large. Further, the attracting eigenlines $((g_n^+)^1)$ and repelling hyperplanes $((g_n^-)^{d-1})$ satisfy 
$$
\lim_{n \rightarrow \infty} \dist_{\Fc_{\alpha_1}}( (g_n^+)^1, x_+^1) = 0 \quad \text{and} \quad \lim_{n \rightarrow \infty} \dist_{\Fc_{\alpha_{d-1}}}( (g_n^-)^{d-1}, x_-^{d-1}) = 0.
$$
Thus
$$
\lim_{n \rightarrow \infty} \dist_{\Fc_{\alpha_1}}( (g_n^+)^1, U_\theta(g_n)^1) = 0 \quad \text{and} \quad \lim_{n \rightarrow \infty} \dist_{\Fc_{\alpha_{d-1}}}( (g_n^-)^{d-1}, U_\theta(g_n^{-1})^{d-1}) = 0.
$$
Next fix unit vectors $(v_n) \subset \Rb^d$ converging to $v$ and  unit vectors $(f_n) \subset \Rb^{d*}$ converging to $f$ such that $v_n \in U_\theta(g_n)^1$ and $\ker f_n = U_\theta(g_n^{-1})^{d-1}$ for all $n$. Then 

\begin{align*}
\lim_{n \rightarrow \infty} & \abs{\omega_1\big(\kappa_\theta(g_n) - \lambda_\theta(g_n) - G_\theta( U_\theta(g_n^{-1}),U_\theta(g_n)\big)}  \\
& = \lim_{n \rightarrow \infty}\abs{ \log\norm{g_n} - \log \lambda_1(g_n)  - \log \frac{ \norm{f_n}^*\norm{v_n} }{\abs{f_n(v_n)}} } \\
& = \lim_{n \rightarrow \infty} \abs{  - \log \lambda_1\left( \frac{g_n}{\norm{g_n}}\right)  + \log \abs{f_n(v_n)} }  \\
& =  \Big| - \log \lambda_1(T) + \log \abs{ f(v)}\Big|= 0. 
\end{align*}

Now if we repeat the above argument with $g_n^{-1}$, we obtain 
\begin{align*}
\lim_{n \rightarrow \infty} & \dist_{\Fc_{\alpha_1}}( (g_n^-)^1, U_\theta(g_n^{-1})^1) = 0, \quad \lim_{n \rightarrow \infty} \dist_{\Fc_{\alpha_{d-1}}}( (g_n^+)^{d-1}, U_\theta(g_n)^{d-1}) = 0,
\end{align*}
and 
\begin{align*}
0 = & \lim_{n \rightarrow \infty}  \abs{\omega_{1}\big(\kappa_\theta(g_n^{-1}) - \lambda_\theta(g_n^{-1}) - G_\theta( U_\theta(g_n),U_\theta(g_n^{-1})\big)} \\ 
 = &  \lim_{n \rightarrow \infty}  \abs{\omega_{d-1}\big(\kappa_\theta(g_n) - \lambda_\theta(g_n) - G_\theta( U_\theta(g_n^{-1}),U_\theta(g_n)\big)}.
\end{align*}
Thus the conclusion of the lemma holds for $n$ large and we have a contradiction. 

Next, we observe that the intermediate case of the lemma when $\Gsf$ is general and $\theta=\{\alpha,\iota^*(\alpha)\}$ 
for some $\alpha\in\Delta$ follows from the special case and Proposition~\ref{prop:reduction to the linear case}. 
Let $R(K,\epsilon)$ be the constant obtained for the compact set $K$ and $\epsilon>0$ in
our special case.
Let $\Phi_\alpha \colon \mathsf G\to\SL(d_\alpha,\Rb)$ and $\xi_\alpha \colon \Fc_\theta\to\Fc_{1,d_\alpha-1}(\Rb^d)$  be given by Proposition~\ref{prop:reduction to the linear case}, 
and choose $C_\alpha>1$ such that $\xi_\alpha$ is $C_\alpha$-biLipschitz. 
Then $R_\alpha(K,\epsilon)$ can be taken to be $R((\xi_\alpha\times \xi_\alpha)(K),\frac{\epsilon}{C_\alpha}))$.

Finally, note that as a consequence of the intermediate case, the most general case of the proposition holds for $R(K,\epsilon):=\max_{\alpha \in \theta}R_\alpha(K,\epsilon)$.
\end{proof}

\begin{lemma}\label{lem:length versus magnitude 2} 
Suppose that $\Gamma \subset \Gsf$ is $\Psf_\theta$-transverse. For any $\epsilon > 0$ and compact subset $K \subset \Fc_\theta^{(2)}$ there exists a finite set $F\subset \Gamma$ such that: if $\gamma \in \Gamma \smallsetminus F$ is $\theta$-loxodromic and 
$(\gamma^{-}, \gamma^+) \in K$, then  
\begin{enumerate}
\item $\dist_{\Fc_\theta}(U_\theta(\gamma^{\pm 1}), \gamma^\pm) < \epsilon$,
\item $\abs{\omega_\alpha\big(\kappa_\theta(\gamma) - \lambda_\theta(\gamma) - G_\theta( \gamma^-,\gamma^+\big)} < \epsilon$ for all $\alpha\in\theta$. 
\end{enumerate} 
\end{lemma}

\begin{proof}  Suppose not. Then there exists an escaping sequence $(\gamma_n)$ of $\theta$-loxodromic elements in $\Gamma$
where $(\gamma_n^{-}, \gamma_n^+) \in K$ 
and at least one of the two conclusions of the lemma fails for each $\gamma_n$. 

Passing to a subsequence, we can suppose that $\gamma_n^\pm \rightarrow a^\pm$ and $U_\theta(\gamma_n^{\pm 1}) \rightarrow b^\pm$. 
We claim that either $a^+=b^+$ or $a^+=b^-$. Indeed, since $a^+,b^-\in\Lambda_\theta(\Gamma)$, if $a^+\neq b^-$, then $a^+$ is transverse to $b^-$, so there is a compact subset $K' \subset \Fc_\theta \smallsetminus \mathcal Z_{b^-}$ such that $\gamma_n^+\in K'$ for all sufficiently large $n$. 
Then Proposition~\ref{prop:characterizing convergence in general} implies that $\gamma_n^+=\gamma_n(\gamma_n^+) \rightarrow b^+$, 
so $a^+=b^+$. Similarly, $a^-=b^-$ or $a^-=b^+$. Since $(\gamma_n^-,\gamma_n^+)\in K$ for all $n$, 
we deduce that $a^-\neq a^+$, and so $b^-\neq b^+$ and $\{a^+,a^-\}=\{b^+,b^-\}$. 
This implies that there is a compact set in $\Fc_\theta^{(2)}$ that contains $(U_\theta(\gamma_n),U_\theta(\gamma_n^{-1}))$ for all large $n$. Thus, by Lemma \ref{prop:length versus magnitude}, for large $n$, the first conclusion holds for $\gamma_n$, so the second conclusion necessarily fails.

Next fix a compact subset $K_0' \subset \Fc_\theta^{(2)}$ with $K \subset {\rm int}(K_0')$, and let 
$$
K' : = \{ (x,y) : (x,y) \in K_0' \text{ or } (y,x) \in K_0'\}.
$$
Also, fix $\alpha\in\theta$. Since $(a^-,a^+) \in K$, for $n$ large we have $(U_\theta(\gamma_n^{-1}), U_\theta(\gamma_n)) \in K'$. Hence, by Lemma~\ref{prop:length versus magnitude}, when $n$ is large we have that 
\[\abs{\omega_\alpha\big(\kappa_\theta(\gamma_n) - \lambda_\theta(\gamma_n) - G_\theta( U_\theta(\gamma_n^{-1}),U_\theta(\gamma_n)\big)} < \frac{\epsilon}{2}.\]
Since $G_\theta$ is continuous, when $n$ is large we also have that 
\[\abs{\omega_\alpha(G_\theta( U_\theta(\gamma_n^{-1}),U_\theta(\gamma_n)\big)-G_\theta( \gamma_n^-,\gamma_n^+)\big)}<\frac{\epsilon}{2}\]
Together, these imply that the second conclusion holds for large $n$, which is a contradiction.
\end{proof} 

\begin{remark} Unlike the case of Lemma \ref{prop:length versus magnitude}, the conclusions of Lemma \ref{lem:length versus magnitude 2} do not hold if we allow $\gamma$ to vary in $\mathsf G$ instead of a $\Psf_\theta$-transverse subgroup $\Gamma\subset\mathsf G$. For instance, if $\Gsf = \SL(4,\Rb)$, $\theta = \{ \alpha_1,\alpha_{3}\}$, and  
$$
\gamma_n: = \begin{pmatrix} n^{1/2} & 0 & 0 & 0 \\ 0 & 1 & n & 0 \\ 0 & 0 & 1 & 0 \\ 0 & 0 & 0 & n^{-1/2} \end{pmatrix},
$$
then $\gamma_n^+ = ( [e_1], {\rm span}\{e_1,e_2,e_3\})$ and $\gamma_n^- = ( [e_4], {\rm span}\{e_2,e_3,e_4\})$ for all $n$, so $(\gamma_n^+,\gamma_n^-)$ are uniformly transverse. However, $U_\theta(\gamma_n) \rightarrow ([e_2], {\rm span}\{e_1,e_2,e_3\})$. 
\end{remark}

\subsection{The equidistribution step}\label{sec:equidistribution step}
Fix a $\Psf_\theta$-transverse subgroup $\Gamma\subset\mathsf G$ and $\phi\in\mathfrak a_\theta^*$ is such that 
$(\Gamma,\phi)$ is of divergence type and the total mass $\norm{m}$ of the $\phi$-BMS measure $m$ for $\Gamma$ is finite.

Let $\delta:=\delta^\phi_\Gamma$. Then let $\mu$ and $\bar\mu$ denote the $\phi$- and $\bar\phi$-Patterson--Sullivan measures for $\Gamma$ of dimension $\delta$.

For any $R>0$, let $\nu_R$ be the Borel measure on $\Fc_\theta^2$ given by
$$
\nu_R := \delta e^{-\delta R}  \sum_{\substack{\gamma \in \Gamma \\ \phi(\kappa_\theta(\gamma)) \leq R}}\mc D_{U_\theta(\gamma^{-1})}\otimes \mc D_{U_\theta(\gamma)},$$ 
where for any $F\in\Fc_\theta$, we write $\mathcal D_F$ to denote the Dirac measure on $\Fc_\theta$ supported at the point $F$. 

We will make use of Theorem \ref{old equidistribution} to establish the following.

\begin{proposition}
\label{off the diagonal}
If $f \colon \Fc_\theta^2 \rightarrow \Rb$ is continuous and compactly supported in $\Fc_\theta^{(2)}$, then 
\[
 \lim_{R \to \infty}\int f\, d\nu_R= \frac{1}{\norm{m}}	 \int f \,d(\bar\mu \otimes \mu).
 \]
 \end{proposition} 
 
 \begin{proof} 
 We will first set up some notation that we will use in this proof. First, the metric $\dist_{\Fc_\theta}$ on $\Fc_\theta$ induces a product metric on $\mathcal F_\theta^{(2)}$, which we also denote by $\dist_{\Fc_\theta}$, i.e.
\[\dist_{\Fc_\theta}((x_1,x_2),(y_1,y_2))=\max\big\{\dist_{\Fc_\theta}(x_1,y_1),\dist_{\Fc_\theta}(x_2,y_2)\big\}\]
for all $(x_1,x_2),(y_1,y_2)\in\Fc_\theta^{(2)}$. Second, since $\{\omega_\alpha|_{\mathfrak a_\theta}:\alpha\in\theta\}$ is a basis for $\mfa_\theta^*$, for each $\alpha\in\theta$, there is a unique $c_\alpha\in\Rb$ such that $\phi=\sum_{\alpha\in\theta}c_\alpha\omega_\alpha$. We define 
\[\abs{\phi}:=\sum_{\alpha\in\theta}\abs{c_\alpha}.\]
 
 To prove the lemma, it suffices to consider the case where $f \geq 0$. 
 Fix $\epsilon > 0$. Since $f$ is compactly supported and the $\theta$-Gromov product $G_\theta$ is continuous on $\Fc_\theta^{(2)}$, there is a finite open cover $\mathcal U$ of the support of $f$ with the property that for each $U\in\mathcal U$, we have $\overline{U} \subset \Fc_\theta^{(2)}$ and there exists $X_U \in \mfa_\theta$ such that
 \begin{equation}
 \label{support assump}
\Big |\omega_\alpha\big(G_\theta(x,y) - X_U\big)\Big| < \epsilon
\end{equation}
 for all $(x,y) \in U$ and $\alpha\in\theta$. Then by using a partition of unity subordinate to $\mathcal U$, we can write $f=\sum_{U \in \mathcal U} f_U$, where each $f_U$ is a continuous non-negative function whose support lies in $U$.

Fix $\epsilon_1 \in (0,\epsilon)$. Since $\mathcal U$ is finite and each $f_U$ is compactly supported (and therefore uniformly continuous), there exists $\beta\in(0,\epsilon)$ such that for all $U\in \mathcal U$, if $d_{\Fc_\theta}\big((x_1,x_2),(y_1,y_2)\big)< \beta$, then 
 \begin{equation}
 \label{function assump} 
 \abs{ f_U( x_1,x_2)-f_U(y_1,y_2)} < \epsilon_1.
 \end{equation} 
 Further shrinking $\beta$, we can assume that 
 $$
\overline{B_\beta(\supp(f_U))} \subset U
$$
where $B_\beta(\supp(f_U))$ is the $\beta$-neighborhood of $\supp(f_U)$. Since $f$ is compactly supported,  by Lemmas \ref{prop:length versus magnitude} and \ref{lem:length versus magnitude 2}, there exists a finite set $F \subset \Gamma$ such that both of the following hold for all $\gamma \in \Gamma \smallsetminus F$:
\begin{itemize}
\item[(i)] if $(U_\theta(\gamma^{-1}),U_\theta(\gamma)) \in {\rm supp}(f)$, then $\gamma$ is $\theta$-loxodromic, 
and
$$\dist_{\Fc_\theta}\big((U_\theta(\gamma),U_\theta(\gamma^{-1})),(\gamma^+,\gamma^-)\big)< \beta\quad\text{and}\quad \abs{\phi(\kappa_\theta(\gamma) - \lambda_\theta(\gamma) - G_\theta( U_\theta(\gamma^{-1}),U_\theta(\gamma)))} < \beta\abs{\phi}.$$
\item[(ii)] if $\gamma$ is $\theta$-loxodromic and $(\gamma^-,\gamma^+)\in{\rm supp}(f)$, then 
$$\dist_{\Fc_\theta}\big((U_\theta(\gamma),U_\theta(\gamma^{-1})),(\gamma^+,\gamma^-)\big)< \beta\quad\text{and}\quad \abs{\phi(\kappa_\theta(\gamma) - \lambda_\theta(\gamma) - G_\theta( \gamma^-,\gamma^+))} < \beta\abs{\phi}.$$
\end{itemize}

Fix $U\in \Uc$ and let $C_U:=-\phi(X_U)+2\epsilon\abs{\phi}$.
For any $R>0$, notice that if $\gamma\in \Gamma\smallsetminus F$, $(U_\theta(\gamma^{-1}),U_\theta(\gamma)) \in {\rm supp}(f_U)$,
 and $ \phi(\kappa_\theta(\gamma)) \leq R$, then (i) and Inequality \eqref{support assump} imply that
\begin{equation}
\label{lambda bound}
\phi (\lambda_\theta(\gamma)) <  R-\phi(X_U) + (\epsilon+\beta) \abs{\phi}< R+C_U.
\end{equation}
Similarly, (ii) and Inequality \eqref{support assump} together  imply that if $\gamma\in \Gamma\smallsetminus F$ is $\theta$-loxodromic, $(\gamma^-,\gamma^+) \in {\rm supp}(f_U)$, 
 and $ \phi(\lambda_\theta(\gamma)) \leq R+\phi(X_U)-2\epsilon\abs{\phi}$, then
\begin{equation}
\label{kappa bound}
\phi (\kappa_\theta(\gamma)) <R. 
\end{equation}

Let $g_U\colon \mathcal F_\theta^{(2)}\to [0,1]$ be a  continuous  function where $\supp(g_U) \subset U$ and  $g_U|_{B_\beta(\supp(f_U))} \equiv 1$. We may then  apply Inequalities \eqref{function assump}, \eqref{lambda bound} and (i) to see that
\begin{equation}
\label{second term bound}
\sum_{\substack{\gamma \in \Gamma \smallsetminus F \\ \phi(\kappa_\theta(\gamma)) \leq R}} f_U ( U_\theta(\gamma^{-1}), U_\theta(\gamma) ) \leq 
\sum_{\substack{\gamma \in \Gamma_{\rm lox} \\ \phi(\lambda_\theta(\gamma)) \leq  R+C_U}} \left(f_U +\epsilon_1 g_U\right)( \gamma^-, \gamma^+ ).
\end{equation}
Similarly, \eqref{function assump}, \eqref{kappa bound} and (ii) imply that
\begin{equation}
\label{second term bound 2}
\sum_{\substack{\gamma \in \Gamma_{\rm lox}\smallsetminus F \\ \phi(\lambda_\theta(\gamma)) \leq  R-C_U}} (f_U-\epsilon_1g_U) ( \gamma^-, \gamma^+ ) \leq \sum_{\substack{\gamma \in \Gamma \\ \phi(\kappa_\theta(\gamma)) \leq R}}  f_U( U_\theta(\gamma^{-1}), U_\theta(\gamma) ) .
\end{equation}

Since $F$ is finite, 
$$
\limsup_{R\to\infty}\int f_U d\nu_R=\limsup_{R\to\infty}\, \delta e^{-\delta R}\sum_{\substack{\gamma \in \Gamma\smallsetminus F \\ \phi(\kappa_\theta(\gamma)) \leq R}} f_U( U_\theta(\gamma^{-1}), U_\theta(\gamma) ).
$$
Then by Inequality \eqref{second term bound},
\begin{align*}
\limsup_{R\to\infty} & \int f_U d\nu_R \le \limsup_{R\to\infty}\, \delta e^{-\delta R}\sum_{\substack{\gamma \in \Gamma_{\rm lox} \\ \phi(\lambda_\theta(\gamma)) \leq  R+C_U}} \left(f_U +\epsilon_1 g_U\right)( \gamma^-, \gamma^+ ) \\
& = e^{\delta C_U}\limsup_{R\to\infty}\, \delta e^{-\delta R}\sum_{\substack{\gamma \in \Gamma_{\rm lox} \\ \phi(\lambda_\theta(\gamma)) \leq  R}} \left(f_U +\epsilon_1 g_U\right)( \gamma^-, \gamma^+ ).
\end{align*}
So by Theorem \ref{old equidistribution},
\begin{align*}
\limsup_{R\to\infty} & \int f_U d\nu_R \le \frac{e^{\delta C_U}}{\norm{m}} \int e^{\delta\phi(G_\theta(x,y))}\left( f_U+\epsilon_1 g_U \right)(x,y)\,d\bar\mu(x)d\mu(y).
\end{align*}
Then using Inequality \eqref{support assump}, the fact that $C_U=-\phi(X_U)+2\epsilon\abs{\phi}$, and the fact that $g_U$ takes values in $[0,1]$ and its support lies in $U$, 
\begin{align}\label{main estimate}
\limsup_{R\to\infty} & \int f_U d\nu_R \le \frac{e^{3\delta\epsilon\abs{\phi}}}{\norm{m}}\int \left( f_U+\epsilon_1 g_U \right)(x,y)\,d\bar\mu(x)d\mu(y)\\
&\le \frac{e^{3\delta\epsilon\abs{\phi}}}{\norm{m}}\left(\int f_U (x,y)\,d\bar\mu(x)d\mu(y)+\epsilon_1\bar\mu\otimes\mu(U)\right). \nonumber
\end{align}

A similar argument, using Inequality \eqref{second term bound 2} in place of \eqref{second term bound}, gives
\begin{align}\label{main estimate 2}
\begin{split}
\liminf_{R\to\infty}\int f_U d\nu_{R}\ge \frac{e^{-3\delta\epsilon\abs{\phi}}}{\norm{m}}\left(\int f_U(x,y)\,d\bar\mu(x)d\mu(y)-\epsilon_1\bar\mu\otimes\mu(U)\right).
\end{split}
\end{align}

Now taking the limit $\epsilon_1\to 0$ in Inequalities \eqref{main estimate} and \eqref{main estimate 2}, summing up over $U\in \mathcal U$, and then taking the limit as $\epsilon \to 0$ provides the estimate 
\[\limsup_{R\to\infty}\int fd\nu_R\le \frac{1}{\norm{m}}\int f d(\bar\mu\otimes \mu)\quad\text{and}\quad\liminf_{R\to\infty}\int fd\nu_R\ge \frac{1}{\norm{m}}\int f d(\bar\mu\otimes \mu),\]
thus proving the lemma.
\end{proof}

 \subsection{Controlling the diagonal} We continue to assume the assumptions and use the notation introduced at the start of the Section~\ref{sec:equidistribution step}. 
 
 In this subsection, we control the $\nu_R$-measure of the non-transverse locus $D:=\Fc_\theta^2 \smallsetminus \Fc_\theta^{(2)}$ in the following sense. When $\theta = \iota^*(\theta)$, we may think of $D$ as a thickened diagonal.

 \begin{proposition}
 \label{on the diagonal}
For any $\epsilon > 0$ there exists an open neighborhood $\Oc\subset\Fc_\theta^2$ of $D$ such that 
\[
 \limsup_{R \to \infty} \nu_R(\Oc) \leq \epsilon. 
 \]
 \end{proposition} 

We start by verifying that $\limsup_{R\to\infty}\nu_R$ satisfies the following  coarse invariance property. 
 
 \begin{lemma}\label{lem:weak invariance} For any $g,h \in \Gamma$, there exist $C_1, C_2 > 0$ such that: if $U, V \subset \Fc_\theta^2$ are open and $\overline{U} \subset V$, then 
 $$
  \limsup_{R \to \infty} \nu_R(U) \leq C_1   \limsup_{R \to \infty} \nu_{R+C_2}((g,h) \cdot V).
  $$
  \end{lemma} 
  
  \begin{proof}  We first observe that there exists a finite set $F \subset \Gamma$ such that: if $\gamma \in \Gamma \smallsetminus F$ and $(U_\theta(\gamma^{-1}), U_\theta(\gamma)) \in U$, 
 then $(U_\theta((h\gamma g^{-1})^{-1}), U_\theta(h \gamma g^{-1})) \in (g,h) \cdot V$. Indeed, if this were not the case, then there would be an infinite sequence $(\gamma_n)$ in 
 $\Gamma$ such that $(U_\theta(\gamma_n^{-1}), U_\theta(\gamma_n)) \in U$ and $(U_\theta((h\gamma_n g^{-1})^{-1}), U_\theta(h \gamma_n g^{-1})) \notin (g,h) \cdot V$ for all $n$. 
On the other hand, by Corollary~\ref{cor:equivariance of U_theta},
    $$
\lim_{n \rightarrow \infty} \dist_{\Fc_\theta}(hU_\theta(\gamma_n), U_\theta(h\gamma_ng^{-1})) = 0= \lim_{n \rightarrow \infty} \dist_{\Fc_\theta}(gU_\theta(\gamma_n^{-1}), U_\theta(g\gamma_n^{-1}h^{-1})).
$$
Since $(U_\theta(\gamma_n^{-1}), U_\theta(\gamma_n)) \in U$, it follows that $\lim_{n\to\infty}(gU_\theta(\gamma_n^{-1}), hU_\theta(\gamma_n)) \in (g,h)\cdot\overline U$.
This implies that  $(U_\theta((h\gamma_n g^{-1})^{-1}), U_\theta(h \gamma_n g^{-1})) \in (g,h) \cdot V$ for sufficiently large $n$, 
which is a contradiction.
  
  By Corollary~\ref{subadditive} there exists $C_2 > 0$ such that 
  $$
  \phi( \kappa(h \gamma g^{-1})) \leq \phi(\kappa(\gamma) )+ C_2
  $$
  for all $\gamma \in \Gamma$. Thus 
  \begin{align*}
  & \limsup_{R \to \infty} \,  \nu_R(U)  =   \limsup_{R \to \infty} \, \delta e^{-\delta R}  \#\left\{ \gamma \in \Gamma : (U_\theta(\gamma^{-1}), U_\theta(\gamma)) \in U \text{ and } \phi(\kappa(\gamma) ) \leq R\right\} \\
 & \quad  = \limsup_{R \to \infty} \, \delta e^{-\delta R}  \#\left\{ \gamma \in \Gamma \smallsetminus F: (U_\theta(\gamma^{-1}), U_\theta(\gamma)) \in U \text{ and } \phi(\kappa(\gamma) ) \leq R\right\}\\
  & \quad \leq \limsup_{R \to \infty} \, \delta e^{-\delta R}\#\left\{ \gamma \in \Gamma : (U_\theta(\gamma^{-1}), U_\theta(\gamma)) \in (g,h) \cdot V \text{ and } \phi(\kappa(\gamma) ) \leq R+C_2\right\} \\
  & \quad \leq \limsup_{R \to \infty} e^{\delta C_2} \nu_{R+C_2}((g,h) \cdot V). \qedhere
  \end{align*} 
  \end{proof} 
  
  We will also use the following observation. 
 
 \begin{lemma}\label{lem:shifting}  Let $\eta_1,\eta_2\in \Gamma$ be $\theta$-loxodromic elements that do not have a common fixed point in $\Lambda_\theta(\Gamma)$. There exists some $r_0 > 0$ such that: if $U \subset \Fc_\theta$ has diameter at most $r_0$ and $U \cap \Lambda_\theta(\Gamma) \neq \varnothing$, then either  
 $$
 U \times \eta_1 U \subset \mathcal F_\theta^{(2)} \quad \text{or} \quad  U \times \eta_2 U  \subset \mathcal F_\theta^{(2)}
 $$
 \end{lemma} 
 
 \begin{proof} Suppose not. Then for each $n$ there exist a subset $U_n\subset\Fc_\theta$ and $x_n,y_n,x_n',y_n' \in U_n$ such that $U_n \cap \Lambda_\theta(\Gamma) \neq \varnothing$, ${\rm diam} \, U_n \leq n^{-1}$, and $(x_n, \eta_1 x_n'), (y_n,\eta_2 y_n') \in  \Fc_\theta^2 \smallsetminus  \Fc_\theta^{(2)}$. By passing to a subsequence, we may assume that $x_n, y_n,x_n',y_n' \rightarrow x \in \Lambda_\theta(\Gamma)$. Then $\eta_1 x, \eta_2 x \in \Lambda_\theta(\Gamma)$ are both non-transverse to $x$. Then since $\Gamma$ is transverse, we must have $x = \eta_1 x = \eta_2 x$, which contradicts the assumption that $\eta_1,\eta_2$ have no common fixed points in $\Lambda_\theta(\Gamma)$. \end{proof} 
 
Finally, we prove Proposition \ref{on the diagonal}.
 
 \begin{proof}[Proof of Proposition \ref{on the diagonal}] 
We start by fixing some constants. Let $d$ be the covering dimension of $\mathcal F_\theta$. Then any open cover of any closed subset of $\mathcal F_\theta$
has a refinement so that no point lies in more than $d+1$ sets in said refinement. 
Fix $\theta$-loxodromic elements $\eta_1,\eta_2\in\Gamma$ that do not have a common fixed point in $\Lambda_\theta(\Gamma)$, and let $r_0$ be the constant given by Lemma \ref{lem:shifting}. Then by Lemma \ref{lem:weak invariance}, there exist $C_1, C_2 > 0$ such that if $U, V \subset \Fc_\theta^2$ are open and $\overline{U} \subset V$, then for both $i=1,2$, 
 $$
  \limsup_{R \to \infty} \nu_R(U) \leq C_1   \limsup_{R \to \infty} \nu_{R+C_2}((\id,\eta_i) \cdot V).
  $$
Since $\mu$ is a Patterson--Sullivan measure, there exists $C_3 > 0$ such that 
$$
\mu(\eta_i E) \leq C_3 \mu(E)
$$
for any Borel subset $E \subset \Fc_\theta$ and $i=1,2$. Since $\mu$ and $\bar \mu$ have no atoms (see \cite[Thm.\ 1.4]{CZZ3}) and $\Fc_\theta$ is compact, there exists $r_1 > 0$ such that  if $U \subset \Fc_\theta$ has diameter at most $r_1$, then 
$$\bar \mu(U) \leq \frac{\epsilon}{2C_1C_3(d+1)}\quad\text{and}\quad\mu(U) \leq \|m\|.$$ 
 
Since $\Lambda_\theta(\Gamma)$ is compact, we may fix a finite cover $\{ V_1,\dots,V_k\}$ of $\Lambda_\theta(\Gamma)$ by open sets in $\Fc_\theta$ 
each with diameter less than $\min\{r_0,r_1\}$. By replacing this open cover with a refinement, 
we can assume that no point of $\Lambda_\theta(\Gamma)$ lies in more than $d+1$ of the $V_n$'s. 
Next for each $V_n$ fix an open set $U_n \subset V_n$ with $\overline{U}_n \subset V_n$ and $\Lambda_\theta(\Gamma) \subset \bigcup_n U_n$. 

Note that $\Oc_0 : = \bigcup_n (U_n \times U_n)\subset\Fc_\theta^2$ is an open neighborhood of 
\[\Lambda_\theta(\Gamma)^2 \smallsetminus \Fc_\theta^{(2)}=\{(x,x)\in\Fc_\theta^2:x\in\Lambda_\theta(\Gamma)\}.\] 
So we can fix an open set $\Oc_1\subset\Fc_\theta^2$ such that $D \subset \Oc_0 \cup \Oc_1$ and $\overline{\Oc}_1 \cap \Lambda_\theta(\Gamma)^2 = \varnothing$. 

We claim that $\Oc := \Oc_0 \cup \Oc_1$ satisfies the lemma. First notice that by the definition of the limit set, there exists a finite set $F \subset \Gamma$ such that if $\gamma \in \Gamma \smallsetminus F$, then 
$$
(U_\theta(\gamma^{-1}), U_\theta(\gamma)) \notin \Oc_1.
$$
Thus $\limsup_{R \to \infty} \nu_R(\Oc_1) =0$, so it suffices to show that $\limsup_{R \to \infty} \nu_R(\Oc_0) <\epsilon$.

By Lemma~\ref{lem:shifting}, for each $n\in\{1,\dots,k\}$ there exists $i_n \in \{1,2\}$ such that $\overline{V}_n \cap \eta_{i_n} \overline{V}_n$ is empty. In particular,  $ \overline{V}_n \times \eta_{i_n} \overline{V}_n$ is compact in $\Fc_\theta^{(2)}$. If $\overline{V}_n \times \eta_{i_n} \overline{V}_n$ has positive $\bar\mu \otimes \mu$ measure, fix a continuous compactly supported function $f_n \colon \Fc_\theta^{(2)} \rightarrow [0,1]$ with $f_n \equiv 1$ on  $\overline{V}_n \times \eta_{i_n} \overline{V}_n$ and 
$$
\int f_n d (\bar \mu \otimes \mu) \leq 2 (\bar \mu \otimes \mu)( \overline{V}_n \times \eta_{i_n} \overline{V}_n).
$$
Otherwise let $f_n \equiv 0$. 

Then 
$$
\sum_n \int f_n d (\bar \mu \otimes \mu) \leq 2\sum_n \bar\mu(\overline{V}_n) \mu(\eta_{i_n} \overline{V}_n) \leq 2 \sum_n  \frac{\epsilon}{2C_1C_3(d+1)} C_3 \mu(\overline{V}_n) \leq \frac{\epsilon}{C_1} \|m\|.
$$
Hence, by Lemma \ref{off the diagonal} and \ref{lem:weak invariance},
\begin{align*}
 \limsup_{R \to \infty} &  \, \nu_R(\Oc_0) \leq \sum_n  \limsup_{R \to \infty} \nu_R(U_n \times U_n) \leq C_1  \sum_n   \limsup_{R \to \infty}\nu_{R+C_2}((\id,\eta_{i_n}) \cdot V_n \times V_n) \\
 & \leq  C_1  \sum_n   \limsup_{R \to \infty} \int f_n d \nu_{R+C_2} = \frac{C_1}{\|m\|} \sum_n \int f_n d (\bar \mu \otimes \mu) \leq \epsilon. \qedhere
 \end{align*}
 \end{proof}

\subsection{Finishing the proof} \label{sec:finishing proof}

We will now combine Propositions \ref{off the diagonal} and \ref{on the diagonal} to prove Theorem \ref{new equidistribution}.

\begin{proof}[Proof of Theorem \ref{new equidistribution}]
As before, we may assume that $f$ is non-negative. For $n\in\Nb$, Proposition \ref{on the diagonal}  implies that there is a neighborhood $\mathcal O_n$ of $D:=\Fc_\theta^2\smallsetminus\Fc_\theta^{(2)}$ such that 
\[\limsup_{R\to\infty}\nu_R(\mathcal O_n)\le\frac{1}{n}.\] 
Since $\Lambda_\theta(\Gamma)$ is transverse, $D \cap \Lambda_\theta(\Gamma)^2 = \{(x,x) : x \in \Lambda_\theta(\Gamma)\}$. 
Hence, since $\mu$ has no atoms (see \cite[Thm.\ 1.4]{CZZ3}), 
$$(\bar \mu\otimes\mu)(D)=\int_{\Lambda_\theta(\Gamma)} \mu(\{x\}) d\bar\mu(x)=0.$$
So by shrinking $\mathcal O_n$ if necessary, we may also assume that for all $n\in\Nb$,
\[(\bar\mu\otimes\mu)(\mathcal O_n)\le\frac{1}{n}.\] 

For each $n\in\Nb$, $\{\Fc_\theta^{(2)},\mathcal O_n\}$ is an open cover of $\Fc_\theta^2$, so we can decompose $f=f_{n,1}+f_{n,2}$ where $f_{n,1}, f_{n,2}$ are both non-negative and have support in $\Fc_\theta^{(2)}$ and $\mathcal O_n$ respectively. Then for every $n\in\Nb$, 
\[\int f_{n,1}\,d\nu_R\le \int f\,d\nu_R =  \int f_{n,1}\,d\nu_R+\int f_{n,2}\,d\nu_R.\]
By taking the limit supremum as $R\to\infty$ and applying Proposition \ref{off the diagonal} to $f_{n,1}$, we have
\[ \limsup_{R\to\infty}\int f\,d\nu_R \le  \frac{1}{\norm{m}}\int f_{n,1}\,d(\bar\mu \otimes \mu)+\limsup_{R\to\infty}\int f_{n,2}\,d\nu_R\]
Also, if $M$ denotes the maximum of $f$ on $\Fc_\theta^2$, then for all $R>0$, 
\[ 0\le \limsup_{R\to\infty}\int f_{n,2}\,d\nu_R\le \frac{M}{n}.\]
Combining the above estimates, we have
\[\limsup_{R\to\infty}\int f\,d\nu_R \le  \frac{1}{\norm{m}}\int f_{n,1}\,d(\bar\mu \otimes \mu)+\frac{M}{n}.\]
Taking the limit as $n\to\infty$ gives
\[\limsup_{R\to\infty}\int f \,d\nu_R \leq \frac{1}{\norm{m}}\int f\,d(\bar\mu \otimes \mu). 
\]

On the other hand, 
\[0\le \int f_{n,2}\,d(\bar\mu \otimes \mu)\le \frac{M}{n}\]
and so
\[ 
\int f\,d(\bar\mu \otimes \mu)- \frac{M}{n} \le \int f_{n,1}\,d(\bar\mu \otimes \mu) .
\]
Hence 
\[ \liminf_{R\to\infty}\int f \,d\nu_R \geq \liminf_{R\to\infty} \int f_{n,1} \,d\nu_R \geq \frac{1}{\|m\|} \int f \,d(\bar\mu \otimes \mu) - \frac{M}{n\|m\|} \]
for all $n$, and so in the limit $n\to\infty$ we have
\[ \liminf_{R\to\infty} \int f\,d\nu_R \geq \frac{1}{\|m\|}  \int f \,d(\bar\mu \otimes \mu) .\]

Hence we obtain
\[ \frac{1}{\norm{m}}\int f\,d(\bar\mu \otimes \mu) \leq \liminf_{R\to\infty}\int f \,d\nu_R \leq \limsup \int f \,d\nu_R \leq \frac{1}{\norm{m}}\int f\,d(\bar\mu \otimes \mu) . \]
Since the two ends of this string of inequalities are equal, we must have equality throughout.
\end{proof}

\end{document}